%% file: Delta-biderivationsOfVirasoroRelatedAlgebras.tex
\documentclass[12pt, reqno]{amsart}
\usepackage{amssymb,dsfont}
\usepackage{eucal}
\usepackage{amsmath}
\usepackage{amscd}
\usepackage[dvips]{color}
\usepackage{multicol}
\usepackage[all]{xy}           
\usepackage{graphicx}
\usepackage{color}
\usepackage{colordvi}
\usepackage{xspace}
\usepackage{txfonts}
\usepackage{lscape}
\usepackage{tikz} 
\usepackage{bm}

\numberwithin{equation}{section}
\usepackage[shortlabels]{enumitem}
\usepackage{ifpdf}
\ifpdf
 \usepackage[colorlinks,final,backref=page,hyperindex]{hyperref}
\else
\fi
\usepackage{tikz}
\usepackage[active]{srcltx}

\makeatletter

\newtheorem{theorem}{Theorem}[section]
\newtheorem{definition}[theorem]{Definition}
\newtheorem{lemma}[theorem]{Lemma}
\newtheorem{proposition}[theorem]{Proposition}
\newtheorem{remark}[theorem]{Remark}

\numberwithin{equation}{section}

\input{newcommands}

\begin{document}
\title[$\delta$-biderivations of Virasoro related algebras]
{$\delta$-biderivations of Virasoro related algebras}
  \author{Chengkang Xu}
  \address{C. Xu: School of Mathematical Sciences, Shangrao Normal Collage, Shangrao, Jiangxi, P. R. China\\ Email: xiaoxiongxu@126.com}

\date{}
\maketitle

\begin{abstract}
 We determine all $\delta$-biderivations for the Witt algebra,
 the Virasoro algebra, the $W$-algebras $W(a,b)$ and their universal central extensions $\widetilde W(a,b)$, and then give some applications.
\end{abstract}

\section{Introduction}

The study of biderivations arises from associative ring theory.
The notion of a (skew-symmetric) biderivation is a valid way to study commuting linear maps of associative rings (see \cite{Bre}).
In \cite{WYC}, the authors gave the definition of a biderivation for Lie algebra,
and studied skew-symmetric biderivations of parabolic subalgebras of finite dimensional simple Lie algebras.
Since then, the studies on biderivations of Lie algebras have attracted many attentions.
Skew-symmetric biderivations were determined
for the Schr\"{o}dinger-Virasoro algebra \cite{WY},
generalized Witt algebras \cite{Chen},
the $W$-algebra $W(a,b)$ \cite{HWX},
the affine-Virasoro algebra of type $A_1$ \cite{WuY},
Kac-Moody algebras \cite{CW},
a loop Schr\"{o}dinger-Witt algebra \cite{LMC},
the Witt algebra and $W(2,2)$ \cite{Tang2017},
a block Lie algebra $\mathcal{B}(q)$ \cite{LGZ},
the twisted Heisenberg-Virasoro algebra \cite{TL},
the mirror Heisenberg-Virasoro algebra \cite{LC}, and so on.
It is also worth to mention that Br$\check{\text{e}}$sar and Zhao in \cite{BZ} introduced a new approach to study skew-symmetric biderivations of Lie algebras using the notion of centroid.
Commuting linear maps were also determined for the aforementioned Lie algebras separately.

Symmetric biderivations are related to commutative post-Lie algebras.
There are also many studies on symmetric biderivations for Lie algebras.
For example, finite-dimensional complex simple Lie algebras \cite{Tang1},
the Witt algebra and the $W$-algebra $W(2,2)$ \cite{Tang2017},
the block Lie algebra $\mathcal{B}(q)$ \cite{LGZ},
the twisted Heisenberg-Virasoro algebra \cite{TL},
the mirror Heisenberg-Virasoro algebra \cite{LC},
the $W$-algebras $W(a,b)$ \cite{Tang2},
higher rank Witt algebras \cite{TY}.
Among the aforementioned Lie algebras,
commutative post-Lie algebra structures were determined
for the Witt algebra, $W(2,2)$, the twisted Heisenberg-Virasoro algebra,
the $W$-algebras $W(a,b)$ and higher rank Witt algebras.

As a generalization of biderivation,
the notion of $\dt$-biderivation was introduced in \cite{YH}, where $\delta$ is an arbitrary complex number.
This notion also generalizes the notion of $\dt$-derivation,
which was first studied by Filippov in \cite{Fi}.
Especially, the $\frac12$-derivation has deep connection with transposed Poisson algebra, which was introduced in \cite{BBGW} as a dual notion of the Poisson algebra.
The authors in \cite{YH} also determined all $\frac12$-derivations of the twisted Heisenberg-Virasoro algebra, the Schr\"odinger-Virasoro algebra,
the extended Schr\"odinger-Virasoro algebra and the twisted Schr\"odinger-Virasoro algebra,
and used them to determine transposed Poisson algebra structures on these algebras.
However, they did not go deeper into computing $\dt$-biderivations of Lie algebras.
Transposed Poisson algebra structures over some Lie algebras were determined recently,
see \cite{KK1, KK2, KLZ, YT} and references therein.
The studies on transposed Poisson algebra structures over Lie algebras in the aforementioned references are all by the means of $\frac12$-derivations.
But the connection between $\frac12$-derivations and transposed Poisson algebra structures on Lie algebras is not direct, but actually through $\frac12$-biderivations.
Therefore, we introduce a notion, called transposed $\dt$-Poisson algebra structure, for Lie algebras, as a generalization of transposed Poisson algebra structure in the last section, and also as an application of $\dt$-biderivations.

In this paper, we determine all $\dt$-biderivations for the Witt algebra,
the Virasoro algebra, the $W$-algebras $\wab$
and their universal central extensions $\twab$.
Let us now briefly describe the organization of this paper.
In Section 2, we recall the definitions of these algebras, and the notions of $\dt$-derivation and $\dt$-biderivation.
In Section 3, we present some basic properties of $\dt$-biderivations.
Section 4 is dedicated to determine $\dt$-biderivations of the Witt algebra and the Virasoro algebra,
and Section 5 is to determine $\dt$-biderivations of the algebras $\wab$ and $\twab$ for any complex numbers $a,b$.
In the last section we provide some applications of $\dt$-biderivations to determine commuting linear maps, commutative post-Lie algebra structures and transposed $\dt$-Poisson algebra structures on the algebras mentioned above.

Throughout this paper, the symbols $\Z$ and $\C$ are to denote the sets of integers and complex numbers respectively.
All vector spaces are assumed to be over $\C$,
and for any Lie algebra $\g$ we denote by $\g'$ the derived subalgebra of $\g$.
Let $G$ be an additive group and $V$ be a $G$-graded space (or Lie algebra) such that
$V=\bigoplus_{g\in G}V_g$.
We call $V_g$ a homogeneous subspace of degree $g$,
and call nonzero elements in $V_g$ are of degree $g$.

\section{Preliminaries}
In this section, we recall the Lie algebras we study in this paper,
the notions of $\dt$-derivation and $\dt$-biderivation,
and some basic results.

\begin{definition}
The \textbf{Witt algebra} is a complex Lie algebra $\W$ with a basis
$\left\{L_{n} \mid  n\in\Z \right\}$
subjecting to the commutation relation
\[[L_m,L_n]=(m-n)L_{m+n}\quad\text{ for any }m,n\in\Z.\]
The \textbf{Virasoro algebra} $\V$ is the universal central extension of $\W$.
Equivalently, $\V$ has a basis
$\left\{L_{n}, C_0 \mid  n\in\Z \right\}$
with commutation relation
\[[L_m,L_n]=(m-n)L_{m+n}+\frac1{12}(m^3-m)C_0\dt_{m+n,0}
   \quad\text{ for any }m,n\in\Z.\]
\end{definition}

The $W$-algebras are tensor products of the Witt algebra and its modules of intermediate series.

\begin{definition}
 Let $a,b\in\C$, the \textbf{$W$-algebra} $\wab$ is a Lie algebra with a $\C$-basis $\{L_n,I_n\mid n\in\Z\}$ satisfying Lie brackets
 \[[L_i,L_j]=(i-j)L_{i+j},\quad [L_i,I_j]=-(a+j+bi)I_{i+j},\quad [I_i,I_j]=0
 \quad\text{for }i,j\in\Z.\]
\end{definition}

Since $\wab\cong W(a+k,b)$ for any $k\in\Z$,
we always assume $0\le a<1$ in this paper.
Note that $\wab$ is perfect if and only if $(a,b)\ne(0,1)$.

The universal central extension $\twab$ of $\wab$ was determined in \cite{GJP}.

\begin{theorem}\label{thm:deftwab}
 The algebra $\twab$ satisfies the following Lie brackets
 \begin{align*}
  &[L_m,L_n]=(m-n)L_{m+n}+\frac{m^3-m}{12}C_0\dt_{m+n,0},\\
  &[L_m,I_n]=-(a+n+bm)I_{m+n}+\dt_{m+n,0}\dt_{a,0}\left(
   (m^2+m)C_1^{0}\dt_{b,0}+(mC_1^{1}+C_2^{1})\dt_{b,1}+
    \frac{m^3-m}{12}C_1^{-1}\dt_{b,-1}
  \right),\\
  &[I_m,I_n]=mC_2^0\dt_{m+n,0}\dt_{a,0}\dt_{b,0}
     +(2m+1)C_2^{\frac12,0}\dt_{m+n+1,0}\dt_{a,\frac12}\dt_{b,0},
 \end{align*}
 where $C_0,C_1^{0},C_1^{1},C_2^{1},C_1^{-1},C_2^0,C_2^{\frac12,0}$ are central elements.
\end{theorem}

Denote by $\mathfrak C$ the center of $\twab$.
Note that $\twab$ is perfect unless $(a,b)=(0,1)$.

\begin{definition}
 Let $\dt\in\C$ and let $\L$ be a Lie algebra.
 A linear map $f:\L\longrightarrow\L$ is called a \textbf{$\dt$-derivation} if
 \[\dt[f(x),y]+\dt[x,f(y)]=f([x,y])\quad\text{ for any }x,y\in\L.\]
\end{definition}

Clearly, when $\dt=1$, the notion of $\dt$-derivation is the ordinary derivation.

Denote by $\derdt\L$ the space of $\dt$-derivations of $\L$ for a fixed $\dt$.
For simplicity we write $\der\L=\mathfrak{Der}_1(\L)$.
All $\dt$-derivations of the Witt algebra $\W$ and the Virasoro algebra $\V$ were determined in \cite{KKS}.

\begin{theorem}[\cite{KKS}]\label{thm:dtderW}
 (1) $\der\W=\spanc{\ad L_j\mid j\in\Z}$.\\
 (2) $\mathcal Der_{\frac12}(\W)=\spanc{\vf_j\mid j\in\Z}$, where $\vf_j(L_i)=L_{i+j}$.\\
 (3) $\derdt\W=0$ if $\dt\neq\frac12,1$.
\end{theorem}

Now we recall the notion of $\dt$-biderivation from \cite{YH}.

\begin{definition}
 Let $\dt\in\C$ and let $\L$ be a Lie algebra.
 A bilinear map $f:\L\times\L\longrightarrow\L$ is called a \textbf{$\dt$-biderivation} of $\L$ if
 \begin{align}
   &\dt[x,f(y,z)]+\dt[f(x,z),y]=f([x,y],z),\label{eq:defdbd1}\\
   &\dt[f(x,y),z]+\dt[y,f(x,z)]=f(x,[y,z])\label{eq:defdbd2}
 \end{align}
 for any $x,y,z\in\L$.
 When $\dt=1$, we simply call $f$ a \textbf{biderivation}.
 A $\dt$-biderivation of $\L$ is called \textbf{central} if $f(x,y)\in Z(\L)$ the center of $\L$,
 called \textbf{symmetric} if $f(x,y)=f(y,x)$,
 and called \textbf{skew-symmetric} if $f(x,y)=-f(y,x)$ for all $x,y\in\L$.
\end{definition}

Denote by $\dbd\L$ the space of $\dt$-biderivations of $\L$ for a fixed $\dt$.
Let $f$ be a $\delta$-biderivation of $\L$, then for any $x\in\L$,
the linear maps $f(x,\cdot)$ and $f(\cdot,x)$ are $\delta$-derivations of $\L$.

\section{Basic properties of $\dt$-biderivations}

In this section we give some basic properties of $\dt$-biderivations over an arbitrary Lie algebra.

\begin{proposition}\label{prop:onecenter}
 Let $\L$ be a Lie algebra with center $Z(\L)$, and $f$ a $\delta$-biderivation of $\L$.\\
 (1) If $\dt\ne0$, then $f(\al,y)=f(y,\al)\in Z(\L)$ for any $\al\in Z(\L), y\in\L$.\\
 (2) $f(\al,y)=f(y,\al)=0$ for any $\al\in Z(\L), y\in[\L,\L]$.\\
 (3) If $\L$ is perfect, then $f(\al,y)=f(y,\al)=0$ for any $\al\in Z(\L), y\in\L$.
\end{proposition}
\begin{proof}
 (1) For any $x\in\L$, $\phi_x=f(x,\cdot)$ is a $\delta$-derivation of $\L$.
 Then for any $y\in\L$, we have
 \[0=\phi_x([\al,y])=\delta[\phi_x(\al),y]+\delta[\al,\phi_x(y)]=\delta[\phi_x(\al),y],\]
 which gives $f(x,\al)=\phi_x(\al)\in Z(\L)$ if $\dt\ne0$.
 Similarly, one gets $f(\al,x)\in Z(\L)$.

 (2) Since $f$ is bilinear, we may assume $y=[z,z']$ for some $z,z'\in\L.$
 Then we have
 \[f(y,\al)=f([z,z'],\al)=\delta[z,f(z',\al)]+\dt[f(z,\al),z']=0.\]
 Similarly, one gets $f(\al,y)=0$. (3) follows from (2).
\end{proof}

\begin{proposition}\label{prop:central}
 Any central $\dt$-biderivation of a perfect Lie algebra is a zero map.
\end{proposition}
\begin{proof}
 Let $\L$ be a perfect Lie algebra and $f$ a central $\dt$-biderivation of $\L$.
 For any $x,y\in\L$, write $y=[z,z']$ for some $z,z'\in\L.$
 By \eqref{eq:defdbd2}, we have
 $f(x,y)=f(x,[z,z'])=\dt[z,f(x,z')]+\dt[f(x,z'),z]=0$.
 So $f=0$.
\end{proof}

\begin{proposition}\label{prop:gradation}
 Suppose $G$ is an abelian group and $\L$ is a finitely generated $G$-graded Lie algebra
 such that $\L=\bigoplus_{g\in G}\L_g$.
 Then $\dbd\L$ is also $G$-graded,
 \[\dbd\L=\bigoplus_{g\in G}\dbd\L_g,\]
 where $\dbd\L_g=\{\phi\in\dbd\L\mid \phi(x,y)\in\L_{g+a+b}\text{ for any } a,b\in G, x\in\L_a,y\in\L_b\}$.
 We call elements in $\dbd\L_g$ are of degree $g$.
\end{proposition}
\begin{proof}
 The proof is similar to that of Lemma 2.1 in \cite{LC}, and we omit the details.
\end{proof}

\begin{proposition}\label{prop:induction}
Let $\phi$ be a $\dt$-biderivation of a Lie algebra $\L$ with center $Z(\L)$.
Define
\[\ov\phi(\ov x, \ov y) = \ov{\phi(x, y)}=\phi(x, y) + Z(\L),\]
for all $x, y\in\L$, where $\ov x = x+Z(\L) \in \L/Z(\L)$.
Then $\ov\phi$ is a $\dt$-biderivation of $\L/Z(\L)$.
\end{proposition}
\begin{proof}
 The verification is easy by using Proposition \ref{prop:onecenter},
 and we omit the details.
\end{proof}

\section{$\delta$-biderivations of the Witt and Virasoro algebras}

In this section we determine all $\delta$-biderivations of the Witt algebra $\W$
and the Virasoro algebra $\V$ in two subsections.

\subsection{$\delta$-biderivations of the Witt algebra}

Clearly the bilinear map $\pi:\W\times\W\longrightarrow\W$ defined by
\[\pi(L_i,L_j)=[L_i,L_j]\quad\text{ for all } i,j\in\Z,\]
is a (1-)biderivation, and for any $n\in\Z$ the bilinear map $\theta_n:\W\times\W\longrightarrow\W$ defined by
\[\theta_n(L_i,L_j)=L_{n+i+j}\quad\text{ for all } i,j\in\Z,\]
is a $\frac12$-biderivation.

\begin{theorem}\label{thm:witt}
 The space of all $\delta$-biderivations of the Witt algebra $\W$ is
 \[\dbd\W=\begin{cases}
     \C\,\pi & \text{ if } \dt=1;\\
     \spanc{\theta_n\mid n\in\Z}& \text{ if } \dt=\frac12;\\
     0 & \text{ otherwise}.\end{cases}\]
\end{theorem}
\begin{proof}
 Let $f$ be a $\dt$-biderivation of $\W$.
 Then for any $i\in\Z$, the linear map $f(L_i,\cdot)$ and $f(\cdot,L_i)$ are $\dt$-derivations of $\W$.
 Moreover, we may assume $f\in\dbd\W_{n}=\{g\in\dbd\W\mid g(L_i,L_j)\in \C L_{n+i+j}\text{ for any }i,j\in\Z\}$ for some $n\in\Z$ by Proposition \ref{prop:gradation}.

 If $\dt\ne\frac12,1$, then by Theorem \ref{thm:dtderW}(3) we have $f(L_i,\cdot)=0$,
 i.e., $f(L_i,L_j)=0$ for any $i,j\in\Z$. Hence $f=0$.

 If $\dt=1$, then for any $i\in\Z$ the linear map $f(L_i,\cdot)$ is a derivation of degree $n+i$, hence $f(L_i,\cdot)=\lmd\ad L_{n+i}$ for some $\lmd\in\C$ by Theorem \ref{thm:dtderW}(1).
 So we have $f(L_i,L_j)=\lmd [L_{n+i},L_j]=\lmd(n+i-j)L_{n+i+j}$.
 Similarly, since $f(\cdot, L_j)$ is a derivation of degree $n+j$, we have
 $f(L_i,L_j)=\mu(n-i+j)L_{n+i+j}$ for some $\mu\in\C$. So
 \[\lmd(n+i-j)=\mu(n-i+j) \quad\text{ for any }i,j\in\Z.\]
 Let $j=n+i$ and we get $n\mu=0$.
 If $n\ne0$, then $\mu=0$, and further we have $\lmd=0$. So $f=0$.
 If $n=0$, then it is clear that $\mu=-\lmd$ and hence $f=\lmd\pi$.
 This proves that $\mathcal{BD}^{[1]}(\W)=\C\,\pi.$

 If $\dt=\frac12$, then for any $i,j\in\Z$ we have by Theorem \ref{thm:dtderW}(2) that
 \[f(L_i,\cdot)=a_i\vf_{n+i},\quad f(\cdot, L_j)=b_j\vf_{n+j},\]
 where $a_i,b_j\in\C$.
 Hence $a_iL_{n+i+j}=f(L_i, L_j)=b_jL_{n+i+j}$,
 which further implies that $a_i=b_j$ for all $i,j\in\Z$.
 So $f$ is a multiple of $\theta_n$.
 This proves the case $\dt=\frac12$.
\end{proof}

\subsection{$\delta$-biderivations of the Virasoro algebra}

In this subsection we determine all $\delta$-biderivations of the Virasoro algebra $\V$.
It is easy to check that the bilinear map $\widetilde\pi:\V\times\V\longrightarrow\V$ defined by
\[\widetilde\pi(L_i,L_j)=[L_i,L_j]=(i-j)L_{i+j}+\frac1{12}(i^3-i)C_0\dt_{i+j,0}\quad\text{ for all } i,j\in\Z,\]
is a (1-)biderivation of degree $0$.

\begin{theorem}\label{thm:vir}
 The space of all $\delta$-biderivations of the Virasoro algebra $\V$ is
 \[\dbd\V=\begin{cases}
     \C\,\widetilde\pi & \text{ if } \dt=1;\\
     0 & \text{ otherwise}.\end{cases}\]
\end{theorem}
\begin{proof}
 Let $f$ be a $\dt$-biderivation of $\V$ and $\ov f$ the $\dt$-biderivation of the quotient algebra $\V/\C C_0$ defined by Proposition \ref{prop:induction}.
 Since $\V$ is perfect, we have $f(C_0,\V)=f(\V,C_0)=0$ by Proposition \ref{prop:onecenter}(3).
 Note that $\V/\C C_0\cong\W$.

 If $\dt\ne\frac12, 1$, then $\ov f=0$ by Theorem \ref{thm:witt},
 i.e., $f$ is a central $\dt$-biderivation of $\V$.
 So $f=0$ by Proposition \ref{prop:central}.

 If $\dt=1$, then $\ov f=\lmd\,\pi$ for some $\lmd\in\C$ by Theorem \ref{thm:witt}.
 So $f_1=f-\lmd \widetilde\pi$ is a central biderivation of $\V$.
 Then $f_1=0$ by Proposition \ref{prop:central}, and hence $f=\lmd\,\widetilde\pi$.

 For the case $\dt=\frac12$, we may further assume that $f$ is of degree $n\in\Z$ by Proposition \ref{prop:gradation}.
 Hence $\ov f=a_n\theta_n$ for some $a_n\in\C$ by Theorem \ref{thm:witt}.
 So for $i,j\in\Z$ we may write
 \[f(L_i,L_j)=a_nL_{n+i+j}+C_{i,j}C_0\quad\text{for some }C_{i,j}\in\C.\]
 For any $i\in\Z\setminus\{0\}$, apply $x=L_i, y=L_{-i}, z=L_{-n}$ to \eqref{eq:defdbd1} and we get
 \[a_n(i^2-1)-24C_{0,-n}=0,\]
 which implies $a_n=0$.
 So $f$ is a central $\frac12$-biderivation of $\V$.
 Hence $f=0$ by Proposition \ref{prop:central}.
\end{proof}

\begin{remark}
 (1) It is well known that a nontrivial derivation of a perfect Lie algebra $\L$ may be uniquely extended to a nontrivial derivation of the universal central extension of $\L$
 (see \cite{BM}).
 Such a extension may not exist for $\frac12$-biderivation.
 From Theorem \ref{thm:witt} and Theorem \ref{thm:vir} we see that the $\frac12$-biderivation $\theta_n$ of $\W$ has no nontrivial extension to the Virsaoro algebra.\\
 (2) One could also prove Theorem \ref{thm:vir} in the same way as in the proof of Theorem \ref{thm:witt} using $\dt$-derivations of $\V$, which were given in \cite{KKS}.
\end{remark}

\section{$\delta$-biderivations of the $W$-algebras}

In this section we determine all $\delta$-biderivations of the $W$-algebras $\wab$ and $\twab$ for arbitrary $\dt,a,b\in\C$.
All $\dt$-derivations of $\wab$ were given in \cite{KKS}.
\begin{proposition}\label{prop:derwab}
 (1) The space of (1-)derivations of $\wab$ is
 \[\der\wab=\begin{cases}
     \inn\wab\oplus\C D_1\oplus \C D_2^0\oplus \C D_3 &\text{if } (a,b)=(0,0),\\
     \inn\wab\oplus\C D_1\oplus \C D_2^k              &\text{if } (a,b)=(0,k), k=1,2,\\
     \inn\wab\oplus\C D_1                             &\text{otherwise},
  \end{cases}\]
 where the derivations $D_1,D_2^k,k=0,1,2$, and $D_3$ are defined by
 \begin{align*}
  &D_1(L_i)=0,\quad D_1(I_i)=I_i;\\
  &D_2^k(L_i)=i^{k+1}I_i,\quad D_2^k(I_i)=0;\\
  &D_3(L_i)=I_i,\quad D_3(I_i)=0.
 \end{align*}
 (2) The space of $\frac12$-derivations of $\wab$ is
 \[\mathfrak{Der}_{\frac12}(\wab)=\begin{cases}
     \spanc{\vf_j,\psi_j\mid j\in\Z}   &\text{if } b=-1,\\
     \C\,\vf_0 &\text{if } b\ne-1,
  \end{cases}\]
 where the $\frac12$-derivations $\vf_j$ and $\psi_j$ are defined by
 \[\vf_j(L_i)=L_{i+j},\ \vf_j(I_i)=I_{i+j};\quad\quad
   \psi_j(L_i)=I_{i+j},\ \psi_j(I_i)=0.\]
 (3) For $\dt\ne\frac12,1$, the space of $\dt$-derivations of $\wab$ is 0.
\end{proposition}

\subsection{$\delta$-biderivations of $\wab$}

In this subsection we determine the space $\dbd\wab$ of all $\delta$-biderivations of the $W$-algebra $\wab$.

\begin{theorem}\label{thm:wab}
 (1) If $\dt\ne\frac12,1$, then $\dbd\wab=0$.\\
 (2) For $\dt=1$, we have
 \[\mathfrak{BD}^{[1]}(\wab)=\begin{cases}
     \spanc{\pi,\psina 0\mid n\in\Z} & \text{ if } b=0;\\
     \spanc{\pi,\psina 1\mid n\in\Z} & \text{ if } b=1;\\
     \spanc{\pi,\Theta^{(0,-1)}}      & \text{ if } a=0,b=-1;\\
     \C\,\pi                         & \text{ otherwise},\end{cases}\]
 where $\pi(x,y)=[x,y], x,y\in\wab$, and the biderivations $\psina 0,\psina 1,\Theta^{(0,-1)}$ are defined by
 \begin{align*}
  &\psina 0(L_i,L_j)=I_{i+j+n},\quad \psina 0(L_i,I_j)=\psina 0(I_i,L_j)
                          =\psina 0(I_i,I_j)=0;\\
  &\psina 1(L_i,L_j)=(a+i+j+n)I_{i+j+n},\quad \psina 1(L_i,I_j)=\psina 1(I_i,L_j)
                          =\psina 1(I_i,I_j)=0;\\
  &\Theta^{(0,-1)}(L_i,L_j)=(i-j)I_{i+j},\quad
   \Theta^{(0,-1)}(L_i,I_j)=\Theta^{(0,-1)}(I_i,L_j)
    =\Theta^{(0,-1)}(I_i,I_j)=0.
 \end{align*}
 (3) For $\dt=\frac12$, we have
 \[\mathfrak{BD}^{[\frac12]}(\wab)=\begin{cases}
     \spanc{\Phi_n,\Psi_n\mid n\in\Z} & \text{ if } b=-1;\\
     0                         & \text{ otherwise},\end{cases}\]
 where $\Phi_n, \Psi_n$ are the bilinear maps defined by
 \begin{align*}
 &\Phi_n(L_i,L_j)=L_{n+i+j},\quad \Phi_n(L_i,I_j)=\Phi_n(I_i,L_j)=I_{n+i+j},\quad
                          \Phi_n(I_i,I_j)=0;\\
 &\Psi_n(L_i,L_j)=I_{n+i+j},\quad \Psi_n(L_i,I_j)=\Psi_n(I_i,L_j)
                          =\Psi_n(I_i,I_j)=0
 \end{align*}
\end{theorem}
\begin{proof}
 Let $f$ be a $\dt$-biderivation of $\wab$.
 Since for any $x\in\wab$, $f(x,\cdot)$ is a $\dt$-derivation of $\wab$,
 statement (1) follows from Proposition \ref{prop:derwab}(3).
 Statement (2) is Theorem 3.2 in \cite{Tang2}.

 (3) If $b\ne-1$, then by Proposition \ref{prop:derwab}(2), for any $x,y\in\wab$, we have
 $f(x,\cdot)=\lmd \vf_0$ and $f(\cdot,y)=\mu\vf_0$ for some $\lmd,\mu\in\C$.
 So $f(x,y)=\lmd y=\mu x$, which forces $\lmd=\mu=0$ and $f=0$.

 Now let $b=-1$ and we may assume that $f$ is of some fixed degree $n\in\Z$ by Proposition \ref{prop:gradation}.
 Then by Proposition \ref{prop:derwab}(2) we may write for $i\in\Z$ that
 \begin{align*}
   &f(L_i,\cdot)=a_{i,n}\vf_{n+i}+b_{i,n}\psi_{n+i},\quad
    f(\cdot, L_i)=a'_{i,n}\vf_{n+i}+b'_{i,n}\psi_{n+i},\\
   &f(I_i,\cdot)=c_{i,n}\vf_{n+i}+d_{i,n}\psi_{n+i},\quad
    f(\cdot,I_i)=c'_{i,n}\vf_{n+i}+d'_{i,n}\psi_{n+i},
 \end{align*}
 where $a_{i,n},b_{i,n},c_{i,n},d_{i,n},a'_{i,n},b'_{i,n},c'_{i,n},d'_{i,n}\in\C$.
 Then we have
 \[f(L_i,L_j)=a_{i,n}L_{n+i+j}+b_{i,n}I_{n+i+j}=a'_{j,n}L_{n+i+j}+b'_{j,n}I_{n+i+j},\]
 which implies that $a'_{j,n}=a_{i,n}, b'_{j,n}=b_{i,n}$ for any $i,j\in\Z$.
 Set $a_n=a_{i,n}, b_{n}=b_{i,n}$ and we have
 \begin{equation}\label{eq:fll}
  f(L_i,L_j)=a_{n}L_{n+i+j}+b_{n}I_{n+i+j}.
 \end{equation}
 Write $f(L_i,I_j)$ in two ways,
 \begin{equation}\label{eq:fli}
  f(L_i,I_j)=a_{n}I_{n+i+j}=c'_{j,n}L_{n+i+j}+d'_{j,n}I_{n+i+j},
 \end{equation}
 which gives $c'_{j,n}=0,d'_{j,n}=a_n$ for any $j\in\Z$.
 Similarly, write $f(I_i,L_j)$ in two ways, and we get $c_{j,n}=0,d_{j,n}=a_n$ for any $j\in\Z$. Hence
 \begin{equation}\label{eq:fil}
  f(I_i,L_j)=a_{n}I_{n+i+j}.
 \end{equation}
 Then $f(\cdot,I_i)=a_{n}\psi_{n+i}$ and $f(I_j,I_i)=0$.
 Together with \eqref{eq:fll}, \eqref{eq:fli} and \eqref{eq:fil}, we see that
 $f=a_n\Phi_n+b_n\Psi_n$.
 This complete the proof.
\end{proof}

\subsection{$\delta$-biderivations of $\twab$}

In this subsection we determine all $\delta$-biderivations of the algebra $\twab$.
The algebra $\widetilde W(0,1)$ has three independent central $\dt$-biderivations $\mathcal A,\mathcal B,\mathcal C$ for any $\dt\in\C$,
which are defined by
\[\mathcal A(I_0,I_0)=C_0,\quad \mathcal B(I_0,I_0)=C_1^1,\quad \mathcal C(I_0,I_0)=C_2^1,\]
and all other values being zero.

\begin{theorem}\label{thm:twab}
 (1) If $\dt\ne \frac12$ or 1, then
 $$\dbd\twab=\begin{cases}
          \spanc{\mathcal A,\mathcal B,\mathcal C}
            & \text{if }(a,b)=(0,1);\\
          \quad\quad  0 & \text{otherwise}.\end{cases}$$
 (2) If $\dt=1$, then
 $$\mathfrak{BD}^{[1]}\left(\twab\right)=\begin{cases}
      \spanc{\widetilde\pi,\widetilde\Theta^{(0,-1)}} & \text{if }a=0, b=-1;\\
      \spanc{\widetilde\pi,\mathcal A,\mathcal B,\mathcal C} &\text{if }a=0, b=1;\\
      \C\,\widetilde\pi & \text{otherwise},
   \end{cases}$$
 where $\widetilde\pi(x,y)=[x,y],x,y\in\twab$ and $\widetilde\Theta^{(0,-1)}$
 is the bilinear map on $\widetilde W(0,-1)$ with only nonvanishing values
 \[\widetilde\Theta^{(0,-1)}(L_i,L_j)=(i-j)I_{i+j}
         +\frac{i^3-i}{12}C_1^{-1}\dt_{i+j,0}\dt_{a,0}.\]
 (3) If $\dt=\frac12$, then
 \[\mathfrak{BD}^{[\frac12]}\left(\twab\right)=\begin{cases}
     \spanc{\widetilde\Psi_n\mid n\in\Z} & \text{ if } 0<a<1,b=-1;\\
     \spanc{\mathcal A,\mathcal B,\mathcal C} & \text{ if } a=0,b=1;\\
       \quad\quad\quad 0               & \text{ otherwise},\end{cases}\]
 where $\widetilde\Psi_n$ is the bilinear map on $\widetilde W(a,-1)$ with only nonvanishing values
 \[\widetilde\Psi_n(L_i,L_j)=I_{n+i+j}.\]
\end{theorem}

Fix a $\dt$-biderivation $f$ of $\twab$ and
let $\ov f$ be the $\dt$-biderivation of the quotient algebra
$\twab/\mathfrak{C}$ defined as in Proposition \ref{prop:induction}.
Note that $\twab/\mathfrak{C}\cong\wab$.
Moreover, note that if $(a,b)\ne(0,1)$, then $\twab$ is perfect and
$f\left(\mathfrak C,\twab\right)=f\left(\twab,\mathfrak C\right)=0$
by Proposition \ref{prop:onecenter}(3).
We split the proof of Theorem \ref{thm:twab} into the following seven lemmas dealing with different cases.
The first one deals with the case $\dt\ne \frac12$ or 1.

\begin{lemma}\label{lem:5.07}
Let $\dt\ne \frac12$ or 1.
Then $f$ is a linear combination of $\mathcal A,\mathcal B,\mathcal C$
if $(a,b)=(0,1)$, and $f=0$ if $(a,b)\ne(0,1)$.
\end{lemma}
\begin{proof}
 By Theorem \ref{thm:wab}(1), we see that $\ov f=0$ and
 $f$ is a central $\dt$-biderivation of $\twab$.
 If $(a,b)\ne(0,1)$, then $\twab$ is perfect and $f=0$ by Proposition \ref{prop:central}.

 Assume $(a,b)=(0,1)$.
 Apply $x=L_i, y=L_j$ to \eqref{eq:defdbd1}, and one gets
 \[12(i-j)f(L_{i+j},z)+(i^3-i)f(C_0,z)\dt_{i+j,0}=0\quad\text{ for any }
   i,j\in\Z,\ z\in\twab,\]
 which implies that $f(C_0,z)=f(L_k,z)=0$ for any $k\in\Z$ and $z\in\twab$.
 Similarly, using \eqref{eq:defdbd2}, we have $f(z,C_0)=f(z,L_k)=0$.

 Let $x=L_i,y=I_{j}$ in \eqref{eq:defdbd1} and we get
 \[-(i+j)f(I_{i+j},z)+\dt_{i+j,0}if(C_1^1,z)+\dt_{i+j,0}f(C_2^1,z)=0\quad\text{for any }i,j\in\Z,z\in\twab,\]
 which implies that $f(I_{i},z)=f(C_1^1,z)=f(C_2^1,z)=0$ for $i\ne0$.
 Similarly, by using \eqref{eq:defdbd2} we get $f(z,I_{i})=f(z,C_1^1)=f(z,C_2^1)=0$ for $i\ne0$.
 In summary, we have shown that $f(x,y)\in\mathfrak C=\spanc{C_0,C_1^1,C_2^1}$
 and $f(x,y)=0$ unless $x=y=I_0$. This proves the lemma.
\end{proof}

The following four lemmas deal with biderivations of $\twab$ (i.e., the case $\dt=1$).

\begin{lemma}\label{lem:5.08}
 If $\dt=1$ and $b=0$, then $f$ is a multiple of $\,\widetilde\pi$.
\end{lemma}
\begin{proof}
 By Theorem \ref{thm:wab}(2), we see that
 \[\ov f=A\pi+\sum_{n\in\Z}B_n\psina 0,\]where $A,B_n\in\C$.
 Therefore, for $i,j\in\Z$ we may write(finite sums)
 \begin{align*}
 &f(L_i,L_j)=A(i-j)L_{i+j}+\sum_{n\in\Z}B_nI_{n+i+j}+A_{i,j},\quad\quad
       f(I_i,I_j)=C_{i,j},\\
 &f(L_i,I_j)=-A(a+j)I_{i+j}+B_{i,j},\quad\quad
  f(I_i,L_j)=A(a+i)I_{i+j}+B'_{i,j},
 \end{align*}
 where $A_{i,j},B_{i,j},B'_{i,j},C_{i,j}\in\mathfrak C$.

 Let $x=L_i,y=L_j,z=L_k$ in \eqref{eq:defdbd1} and we get
 \begin{equation}\label{eqb0:aijk}
 \begin{split}
  12(i-j)A_{i+j,k}=&A(j-k)(i^3-i)C_0\dt_{i+j+k,0}-A(i-k)(j^3-j)C_0\dt_{i+j+k,0}\\
                   &+12B_{-i-j-k}(i^2+i-j^2-j)C_1^0\dt_{a,0}.
 \end{split}
 \end{equation}
 Let $j=0$ and $i\ne0$ in \eqref{eqb0:aijk} and we have
 \begin{equation}\label{eqb0:ai0k}
  A_{i,k}=\frac A{12}(i^3-i)C_0\dt_{i+j+k,0}+B_{-i-k}(i+1)C_1^0\dt_{a,0}
  \quad\text{for }i,k\in\Z, i\ne0.
 \end{equation}
 Let $j=-i\ne0$ in \eqref{eqb0:aijk} and we have
 \[A_{0,k}=B_{-k}C_1^0\dt_{a,0}\quad\text{for }k\in\Z.\]
 This says \eqref{eqb0:ai0k} stands for $i=0$ too, i.e., we have
 \begin{equation*}
  A_{i,j}=\frac A{12}(i^3-i)C_0\dt_{i+j,0}+B_{-i-j}(i+1)C_1^0\dt_{a,0}
  \quad\text{for any }i,j\in\Z.
 \end{equation*}
 In particular, if $a\ne0$, then $A_{i,j}=\frac A{12}(i^3-i)C_0\dt_{i+j,0}$.
 If $a=0$, then apply $x=L_i,y=L_j,z=L_k$ in \eqref{eq:defdbd2}, and similarly we get
 \begin{equation*}
  A_{i,j}=\frac A{12}(i^3-i)C_0\dt_{i+j,0}+B_{-i-j}(j+1)C_1^0
  \quad\text{for any }i,j\in\Z.
 \end{equation*}
 Comparing the above two equations we see that $B_n=0$ for any $n\in\Z$,
 and we still have $A_{i,j}=\frac A{12}(i^3-i)C_0\dt_{i+j,0}$.
 Hence
 \begin{equation}\label{eqb0:flilj}
  f(L_i,L_j)=A[L_i,L_j] \quad\text{for any }i,j\in\Z.
 \end{equation}

 Let $x=L_i,y=L_j,z=I_k$ in \eqref{eq:defdbd1} and we get
 \begin{equation}\label{eqb0:bijk}
  (i-j)B_{i+j,k}=-A(a+k)(i^2+i-j^2-j)C_1^0\dt_{i+j+k,0}\dt_{a,0}.
 \end{equation}
 Let $j=0$ and $i\ne0$ in \eqref{eqb0:bijk} and we have
 \begin{equation*}
  B_{i,k}=A(i^2+i)C_1^0\dt_{i+k,0}\dt_{a,0}.
 \end{equation*}
 Let $j=-i\ne0$ in \eqref{eqb0:bijk} and we get $B_{0,k}=0$. So
 \begin{equation}\label{eqb0:fliij}
  f(L_i,I_j)=A[L_i,I_j]  \quad\text{for any }i,j\in\Z.
 \end{equation}

 Let $x=I_i,y=L_j,z=L_k$ in \eqref{eq:defdbd2} and we get
 \[(j-k)B'_{i,j+k}=iAC_1^0(j^2+j-i^2+i)\dt_{k+i+j,0}\dt_{a,0},\]
 which implies that
 \[B'_{i,j}=-A(i^2-i)C_1^0\dt_{i+j,0}\quad\text{for any }i,j\in\Z.\]
 Therefore,
 \begin{equation}\label{eqb0:fiilj}
  f(I_i,L_j)=A[I_i,L_j]\quad\text{for any }i,j\in\Z.
 \end{equation}

 Let $x=L_i,y=I_j,z=I_k$ in \eqref{eq:defdbd1} and we get
 \[(a+j)C_{i+j,k}=A(a+k)\left(-jC_2^0\dt_{i+j+k,0}\dt_{a,0}
       -(2j+1)C_2^{\frac12,0}\dt_{i+j+k+1,0}\dt_{a,\frac12}\right),\]
 from which we obtain
 \begin{equation}\label{eqb0:fiiij}
   f(I_i,I_j)=C_{i,j}=iAC_2^0\dt_{i+j,0}\dt_{a,0}
   +A(2i+1)C_2^{\frac12,0}\dt_{i+j+1,0}\dt_{a,\frac12}=A[I_i,I_j]
   \quad\text{for any }i,j\in\Z.
 \end{equation}
 From \eqref{eqb0:flilj},\eqref{eqb0:fliij},\eqref{eqb0:fiilj} and \eqref{eqb0:fiiij}, we get that $f=A\widetilde\pi$.
\end{proof}

\begin{lemma}\label{lem:5.09}
 If $\dt=1, a=0, b=-1$, then $f$ is a linear combination of $\,\widetilde\pi$ and $\widetilde\Theta^{(0,-1)}$.
\end{lemma}
\begin{proof}
 In this case we may write by Theorem \ref{thm:wab}(2)
 \[\ov f=A\pi+B\Theta^{(0,-1)},\]
 where $A,B\in\C$. For any $i,j\in\Z$ write
 \begin{align*}
 &f(L_i,L_j)=A(i-j)L_{i+j}+B(i-j)I_{i+j}+A_{i,j},\quad\quad
       f(I_i,I_j)=C_{i,j},\\
 &f(L_i,I_j)=A(i-j)I_{i+j}+B_{i,j},\quad\quad
  f(I_i,L_j)=A(i-j)I_{i+j}+B'_{i,j},
 \end{align*}
 where $A_{i,j},B_{i,j},B'_{i,j},C_{i,j}\in\mathfrak C$.

 Let $x=L_i,y=L_j,z=L_k$ in \eqref{eq:defdbd2} and we get
 \begin{equation*}
 \begin{split}
  12(j-k)A_{i,j+k}=&A(i-k)(j^3-j)C_0\dt_{i+j+k,0}-A(i-j)(k^3-k)C_0\dt_{i+j+k,0}\\
                   &+B(i-k)(j^3-j)C_1^{-1}\dt_{i+j+k,0}\dt_{a,0}
                   -B(i-j)(k^3-k)C_1^{-1}\dt_{i+j+k,0}\dt_{a,0}.
 \end{split}
 \end{equation*}
 Solving this equation in the same way as solving \eqref{eqb0:aijk}, we get
 \begin{equation*}
  A_{i,j}=\frac{i^3-i}{12}\left(AC_0+BC_1^{-1}\dt_{a,0}\right)\dt_{i+j,0}.
 \end{equation*}
 So
 \begin{equation}\label{eqb-1:flilj}
  f(L_i,L_j)=A\,\widetilde\pi(L_i,L_j)+B\,\widetilde\Theta^{(0,-1)}(L_i,L_j)
             \quad\text{for any }i,j\in\Z.
 \end{equation}
 Apply $x=L_i,y=I_j,z=I_k$ in \eqref{eq:defdbd1} and we get
 \[(i-j)C_{i+j,k}=0,\]
 from which we obtain $C_{i,j}=0$ for any $i,j\in\Z$, and
 \begin{equation}\label{eqb-1:fiiij}
  f(I_i,I_j)=0.
 \end{equation}

 Let $x=L_i,y=L_j,z=I_k$ in \eqref{eq:defdbd2} and we get
 \[12(k-j)B_{i,j+k}=A(i-j)(k^3-k)C_1^{-1}\dt_{i+j+k,0}
                   -A(i-k)(j^3-j)C_1^{-1}\dt_{i+j+k,0}.\]
 Solving this equation we have
 \[B_{i,j}=\frac{1}{12}A(i^3-i)C_1^{-1}\dt_{i+j,0},\]
 and hence
 \begin{equation}\label{eqb-1:fliij}
  f(L_i,I_j)=A[L_i,I_j].
 \end{equation}
 Similarly, applying $x=I_i,y=L_j,z=L_k$ in \eqref{eq:defdbd2} and we get
 \begin{equation}\label{eqb-1:fiilj}
  f(I_i,L_j)=A[I_i,L_j].
 \end{equation}
 From \eqref{eqb-1:flilj}, \eqref{eqb-1:fiiij}, \eqref{eqb-1:fliij} and \eqref{eqb-1:fiilj}
 we see that $f=A\,\widetilde\pi+B\,\widetilde\Theta^{(0,-1)}$.
\end{proof}

\begin{lemma}\label{lem:5.10}
 Let $\dt=1, b=1$. Then $f$ is a linear combination of
 $\,\widetilde\pi, \mathcal A, \mathcal B, \mathcal C$ if $a=0$,
 and $f$ is a multiple of $\,\widetilde\pi$ if $a\ne0$.
\end{lemma}
\begin{proof}
 First we mention that although $\widetilde W(0,1)$ is not perfect since
 $$I_0\notin\widetilde W(0,1)'=\left[\widetilde W(0,1),\widetilde W(0,1)\right]
   =\spanc{L_i,I_j,C_0,C_1^1,C_2^1\mid i\in\Z,j\in\Z\setminus\{0\}},$$
 we still have by Proposition \ref{prop:onecenter}(2)
 $$f\(\mathfrak C,\widetilde W(0,1)'\)=f\(\widetilde W(0,1)',\mathfrak C\)=0.$$

 By Theorem \ref{thm:wab}(2), we may write (a finite sum)
 \[\ov f=A\,\pi+\sum_{n\in\Z}B_n\psina 1,\]
 where $A,B_n\in\C$.
 Hence for any $i,j\in\Z$ we have
 \begin{align*}
 &f(L_i,L_j)=A(i-j)L_{i+j}+\sum_{n\in\Z}B_n(a+n+i+j)I_{n+i+j}+A_{i,j},\quad\quad
       f(I_i,I_j)=C_{i,j},\\
 &f(L_i,I_j)=-A(a+i+j)I_{i+j}+B_{i,j},\quad\quad
  f(I_i,L_j)=A(a+i+j)I_{i+j}+B'_{i,j},
 \end{align*}
 where $A_{i,j},B_{i,j},B'_{i,j},C_{i,j}\in\mathfrak C$.

 Let $x=L_i,y=L_j,z=L_k$ in \eqref{eq:defdbd1} and one gets
 \begin{align*}
  (i-j)A_{i+j,k}=&A(j-k)\frac{i^3-i}{12}C_0\dt_{i+j+k,0}
                    -A(i-k)\frac{j^3-j}{12}C_0\dt_{i+j+k,0}\\
        &+B_{-i-j-k}\dt_{a,0}\dt_{b,1}\((j^2-i^2)C_1^1+(j-i)C_2^1\)
        \quad\text{for any }i,j,k\in\Z.
 \end{align*}
 Solving this equation we get
 \begin{equation}\label{eqb1:aij1}
  A_{i,j}=A\frac{i^3-i}{12}C_0\dt_{i+j,0}-B_{-i-j}\dt_{a,0}\dt_{b,1}\(iC_1^1+C_2^1\).
 \end{equation}
 Similarly, putting $x=L_i,y=L_j,z=L_k$ in \eqref{eq:defdbd2} gives that
 \begin{align*}
  (j-k)A_{i,j+k}=&A(i-k)\frac{j^3-j}{12}C_0\dt_{i+j+k,0}
                    -A(i-j)\frac{k^3-k}{12}C_0\dt_{i+j+k,0}\\
        &+B_{-i-j-k}\dt_{a,0}\dt_{b,1}\((k^2-j^2)C_1^1+(k-j)C_2^1\)
        \quad\text{for any }i,j,k\in\Z,
 \end{align*}
 which implies that
 \begin{equation}\label{eqb1:aij2}
  A_{i,j}=A\frac{i^3-i}{12}C_0\dt_{i+j,0}+B_{-i-j}\dt_{a,0}\dt_{b,1}\(iC_1^1-C_2^1\).
 \end{equation}
 Comparing \eqref{eqb1:aij1} and \eqref{eqb1:aij2}, one sees that
 \[B_n=0\quad\text{for any }n\in\Z\text{ if }(a,b)=(0,1).\]
 So we have
 \begin{equation}\label{eqb1:flilj}
  A_{i,j}=\frac{1}{12}A(i^3-i)C_0\dt_{i+j,0}\quad\text{ and }\quad
  f(L_i,L_j)=A[L_i,L_j].
 \end{equation}

 Apply $x=L_i,y=L_j,z=I_k$ to \eqref{eq:defdbd1} and we get
 \[(i-j)B_{i+j,k}+\dt_{i+j,0}\dt_{k,0}\frac{i^3-i}{12}f\(C_0,I_0\)=
  A\dt_{i+j+k,0}\dt_{a,0}\dt_{b,1}\((i^2-j^2)C_1^1+(i-j)C_2^1\)
         \text{ for any }i,j,k\in\Z.\]
 Let $j=0$ and we have
 \[B_{i,k}=A\dt_{i+k,0}\dt_{a,0}\dt_{b,1}\(iC_1^1+C_2^1\)
  \quad\text{for }i,k\in\Z\text{ and } i\ne0. \]
 Take $j=-i\ne0$ and we get
 \[B_{0,k}+\dt_{k,0}\frac{i^2-1}{24}f\(C_0,I_0\)
   =A\dt_{k,0}\dt_{a,0}\dt_{b,1}C_2^1
  \quad\text{for }i,k\in\Z\text{ and } i\ne0. \]
 If $k\ne0$ then $B_{0,k}=0$.
 If $k=0$ the the above equation forces $f\(C_0,I_0\)=0$
 and then $B_{0,0}=A\dt_{a,0}\dt_{b,1}C_2^1$.
 In summary, we have
 \begin{equation}\label{eqb1:fliij}
 f\(C_0,I_0\)=0,\quad
 B_{i,j}=A\dt_{i+j,0}\dt_{a,0}\dt_{b,1}\(iC_1^1+C_2^1\)\quad\text{and}\quad
 f(L_i,I_j)=A[L_i,I_j].
 \end{equation}
 Similarly, apply $x=I_i,y=L_j,z=L_k$ to \eqref{eq:defdbd2} and we get
 \begin{equation}\label{eqb1:fiilj}
 f\(I_0, C_0\)=0,\quad
 B'_{i,j}=-A\dt_{i+j,0}\dt_{a,0}\dt_{b,1}\(jC_1^1+C_2^1\)\quad\text{and}\quad
 f(I_i,L_j)=A[I_i,L_j].
 \end{equation}

 Apply $x=L_i,y=I_j,z=I_k$ to \eqref{eq:defdbd1} and we get
 \[(a+i+j)C_{i+j,k}-\dt_{i+j,0}\dt_{a,0}\dt_{k,0}f\(iC_1^1+C_2^1,I_0\)=0
         \quad\text{for any }i,j,k\in\Z.\]
 If $a\ne0$, then $(a+i)C_{i,k}=0$ for any $i,k\in\Z$.
 If $a=0$, let $j=-i\ne0, k=0$ and we get $if\(C_1^1,I_0\)+f\(C_2^1,I_0\)=0$
 for any $i\in\Z$.
 So $f\(C_1^1,I_0\)=f\(C_2^1,I_0\)=0$,
 and then $iC_{i,k}=0$ for any $i,k\in\Z$.
 In summary we have
 \begin{equation*}\label{eqb1:cij1}
  f\(C_1^1,I_0\)=f\(C_2^1,I_0\)=0,\text{ and }(a+i)C_{i,j}=0\text{ for any }i,j\in\Z.
 \end{equation*}
 Similarly, apply $x=I_i,y=L_j,z=I_k$ to \eqref{eq:defdbd2} and one gets
 \begin{equation*}\label{eqb1:cij2}
  f\(I_0,C_1^1\)=f\(I_0,C_2^1\)=0,\text{ and }(a+j)C_{i,j}=0\text{ for any }i,j\in\Z.
 \end{equation*}
 Comparing the above two equations we conclude that
 \begin{equation}
 \begin{aligned}\label{eqb1:fiiij}
  &f\(C_1^1,I_0\)=f\(C_2^1,I_0\)=f\(I_0,C_1^1\)=f\(I_0,C_2^1\)=0;\\
  &f(I_i,I_j)=C_{i,j}=0\quad\text{unless }(a,b)=(0,1)\text{ and }(i,j)=(0,0).
 \end{aligned}
 \end{equation}
 Now for the case $(a,b)=(0,1)$, apply $x=I_0,y=L_k,z=I_0$ to \eqref{eq:defdbd2}
 and we see that $[C_{0,0},L_k]=0$ for any $k\in\Z$.
 This forces $C_{0,0}\in\mathfrak C=\spanc{C_0,C_1^1,C_2^1}$.
 Combining with \eqref{eqb1:flilj}, \eqref{eqb1:fliij}, \eqref{eqb1:fiilj}
 and \eqref{eqb1:fiiij},
 we proves the lemma.
\end{proof}

\begin{lemma}\label{lem:5.11}
 Let $\dt=1$, and suppose that $b\ne 0,1$ and $a\notin\Z$ if $b=-1$.
 Then $f$ is a multiple of $\,\widetilde\pi$.
\end{lemma}
\begin{proof}
 Clearly, $\widetilde\pi$ is a biderivation of $\twab$, and $\twab$ is perfect.
 By Theorem \ref{thm:wab}(2), we have $\ov f=A\pi$ for some $A\in\C$.
 So we have $f(x,y)=A[x,y]+\mathfrak C$ for any $x,y\in\twab$.
 Therefore, $f_1=f-A\widetilde\pi$ is a central biderivation of $\twab$,
 hence $f_1=0$ by Proposition \ref{prop:central}.
 So $f=A\widetilde\pi$.
\end{proof}

The following two lemmas deal with $\frac12$-biderivations of $\twab$.

\begin{lemma}\label{lem:5.12}
Let $\dt=\frac12$ and $b=-1$.
Then $f$ is a linear combination of $\,\widetilde\Psi_n$'s if $a\ne0$,
and $f=0$ if $a=0$.
\end{lemma}
\begin{proof}
 In this case we have $\ov f=\sum_{n\in\Z}A_n\Phi_n+B_n\Psi_n$ for some $A_n,B_n\in\C$.
 So for any $i,j\in\Z$ we may write
 \begin{align*}
 &f(L_i,L_j)=\sum_{n\in\Z}A_nL_{n+i+j}+\sum_{n\in\Z}B_nI_{n+i+j}+A_{i,j},\quad\quad
       f(I_i,I_j)=C_{i,j},\\
 &f(L_i,I_j)=\sum_{n\in\Z}A_nI_{n+i+j}+B_{i,j},\quad\quad
  f(I_i,L_j)=\sum_{n\in\Z}A_nI_{n+i+j}+B'_{i,j},
 \end{align*}
 where $A_{i,j},B_{i,j},B'_{i,j},C_{i,j}\in\mathfrak C$.

 Apply $x=L_i,y=L_j,z=L_k$ to \eqref{eq:defdbd1} and we get
 \begin{equation}\label{eqb-1:aijk}
 24(i-j)A_{i+j,k}=\(i^3-i-j^3+j\)\(A_{-i-j-k}C_0+B_{-i-j-k}C_1^{-1}\dt_{a,0}\dt_{b,-1}\)
         \quad\text{for any }i,j,k\in\Z.
 \end{equation}
 Take $j=-i\ne0$ and we have
 \[24A_{0,k}=\(i^2-1\)\(A_{-k}C_0+B_{-k}C_1^{-1}\dt_{a,0}\dt_{b,-1}\)
         \quad\text{for any }i,k\in\Z, i\ne0,\]
 which implies that
 \begin{equation}\label{eqb-1:akc0}
 A_{k}C_0+B_{k}C_1^{-1}\dt_{a,0}\dt_{b,-1}=A_{0,k}=0
 \quad\text{for any }k\in\Z.
 \end{equation}
 This turns \eqref{eqb-1:aijk} into $(i-j)A_{i+j,k}=0$, which implies that
 $$A_{i,k}=0\quad\text{for any }i,k\in\Z.$$

 Apply $x=L_i,y=L_j,z=I_k$ to \eqref{eq:defdbd1} and we get
 \begin{equation}\label{eqb-1:bijk}
 24(i-j)B_{i+j,k}=A_{-i-j-k}C_1^{-1}\dt_{a,0}\dt_{b,-1}\(i^3-i-j^3+j\)
         \quad\text{for any }i,j,k\in\Z.
 \end{equation}
 Let $j=-i\ne0$ and we have
 \[24B_{0,k}=A_{-k}C_1^{-1}\dt_{a,0}\dt_{b,-1}\(i^2-1\)
         \quad\text{for any }i,k\in\Z, i\ne0,\]
 which implies that
 \begin{equation}\label{eqb-1:akc1}
 A_{k}C_1^{-1}\dt_{a,0}\dt_{b,-1}=B_{0,k}=0
 \quad\text{for any }k\in\Z.
 \end{equation}
 So we have $A_k=0$ by \eqref{eqb-1:akc1} if $a=0$, and by \eqref{eqb-1:akc0} if $a\ne0$.
 Take $j=0$ in \eqref{eqb-1:bijk} and one can see that $B_{i,k}=0$ for any $i,k\in\Z$.
 Hence we have
 \begin{equation}\label{eqb-1:fliij}
 f(L_i,I_j)=0\ \text{and }f(I_i,L_j)=B'_{i,j}\quad\text{for any }i,j\in\Z.
 \end{equation}
 Then apply $x=I_i,y=L_j,z=L_k$ or $z=I_k$ to \eqref{eq:defdbd1} and we can get
 \[B'_{i,j}=C_{i,j}=0\quad\text{for any }i,j\in\Z.\]
 So we have
 \begin{equation}\label{eqb-1:fii}
 f(I_i,L_j)=f(I_i,I_j)=0.
 \end{equation}
 Moreover, we have $B_k=0$ if $a=0$ by \eqref{eqb-1:akc0}. Therefore
 \[f(L_i,L_j)=\begin{cases}
           \sum_{n\in\Z}B_nI_{n+i+j} & \text{if }a\ne0;\\
           \quad\quad    0       & \text{if }a=0.\end{cases}\]
 Combining with \eqref{eqb-1:fliij} and \eqref{eqb-1:fii} we prove the lemma.
\end{proof}

\begin{lemma}\label{lem:5.13}
 Let $\dt=\frac12$ and $b\ne-1$.
 Then $f$ is a linear combination of $\mathcal A, \mathcal B, \mathcal C$ if $(a,b)=(0,1)$, and $f=0$ if $(a,b)\ne(0,1)$.
\end{lemma}
\begin{proof}
 By Theorem \ref{thm:wab}(3) we see that $f$ is a central $\frac12$-biderivation of $\twab$. Hence $f=0$ by Proposition \ref{prop:central} if $(a,b)\ne(0,1)$.

 In the rest of this proof we assume $(a,b)=(0,1)$.
 Note that the derived subalgebra $\widetilde W(0,1)'$ of $\widetilde W(0,1)$ is perfect, and $\widetilde W(0,1)=\C I_0\oplus\widetilde W(0,1)'$.
 So $f$ is a zero map when restricted to $\widetilde W(0,1)'$.

 Let $x=L_i,y=L_j,z=I_0$ in \eqref{eq:defdbd1} and we have
 \[12(i-j)f(L_{i+j},I_0)+\(i^3-i\)f(C_0,I_0)\dt_{i+j,0}=0.\]
 Take $j=0, i\ne0$ and we see that $f(L_{i},I_0)=0$ for all $i\ne0$.
 Take $j=-i\ne0$ and we get $f(L_0,I_0)=f(C_0,I_0)=0$.
 Similarly, let $x=I_0, y=L_i,z=L_j$ in \eqref{eq:defdbd2} and one can deduce that
 $f(I_0, C_0)=f(I_0,L_{i})=0$ for all $i\in\Z$.

 Let $x=I_i,y=L_j,z=I_0$ in \eqref{eq:defdbd1} and we have
 \[(i+j)f(I_{i+j},I_0)-\dt_{i+j,0}\(if(C_1^1,I_0)+f(C_2^1,I_0)\)=0,\]
 which gives that $f(C_1^1,I_0)=f(C_2^1,I_0)=f(I_{i},I_0)=0$ for all $i\ne0$.
 Similarly, let $x=I_0, y=L_i,z=I_j$ in \eqref{eq:defdbd2} and one can get
 $f(I_0, C_1^1)=f(I_0,C_2^1)=f(I_0,I_{i})=0$ for all $i\ne0$.

 In summary, we have shown that $f(x,y)=0$ for any $x,y\in\widetilde W(0,1)$ unless
 $x=y=I_0$, and that $f(I_0,I_0)\in\mathfrak C$. This proves the lemma.
\end{proof}

It is clear that Theorem \ref{thm:twab} follows from Lemmas \ref{lem:5.07},
\ref{lem:5.08}, \ref{lem:5.09}, \ref{lem:5.10}, \ref{lem:5.11}, \ref{lem:5.12} and \ref{lem:5.13}.

\begin{remark}
 From Theorem \ref{thm:wab} and Theorem \ref{thm:twab} we can see that $\dt$-biderivations of a Lie algebra may not be extended to its universal central extension.
 For examples, the biderivations $\psina 0$ of $W(a,0)$ and $\psina 1$ of $W(a,1)$, the $\frac12$-biderivations $\Phi_n$ of $W(a,-1)$ and $\Psi_n$ of $W(0,-1)$.
\end{remark}

\section{Applications}

In this section we determine all commuting linear maps, commutative post-Lie algebra structures and transposed $\dt$-Poisson algebra structures on the Lie algebras studied in Section 4 and Section 5 as applications of $\dt$-biderivations.
In this section we denote by $\L$ one of the Witt algebra $\W$,
the Virasoro algebra $\V$, the $W$-algebras $\wab$ or $\twab$.

\subsection{Commuting linear maps}

A linear map $\vf$ on a Lie algebra $\g$ is called \textbf{commuting} if
\[[\vf(x),x]=0\quad\text{ for any }x\in\g.\]
This condition is equivalent to
\[[\vf(x),y]=[x,\vf(y)]\quad\text{ for any }x,y\in\g.\]
Clearly, the identity map $\mathrm{id}$ is a commuting linear map of $\g$.

\begin{theorem}
 The space of all commuting linear maps of $\L$ is
 \[\mathfrak{clm}(\L)=\begin{cases}
   \spanc{\mathrm{id}, \Theta_1^{(0,-1)}} & \text{if }\L=W(0,-1);\\
   \spanc{\mathrm{id}, \widetilde\Theta_1^{(0,-1)}} & \text{if }\L=\widetilde W(0,-1);\\
   \spanc{\mathrm{id}} & \text{otherwise},\end{cases}\]
 where the linear maps $\Theta_1^{(0,-1)}$ and $\widetilde\Theta_1^{(0,-1)}$ are defined by
 \begin{align*}
 &\Theta_1^{(0,-1)}(L_i)=I_i,\quad \Theta_1^{(0,-1)}(I_i)=0;\\
 &\widetilde\Theta_1^{(0,-1)}(L_i)=I_i,\quad  \widetilde\Theta_1^{(0,-1)}(I_i)=0,\quad
    \widetilde\Theta_1^{(0,-1)}(C_0)=\widetilde\Theta_1^{(0,-1)}(C_1^{-1})=0.
 \end{align*}
\end{theorem}
\begin{proof}
 Let $\vf$ be a commuting linear map of $\L$,
 the bilinear map $f_\vf$ on $\L$ defined by
 \begin{equation}\label{eq:clm}
  f_\vf(x,y)=[\vf(x),y]=[x,\vf(y)]\quad\text{for any }x,y\in\L,
 \end{equation}
 is a skew-symmetric biderivation of $\L$.
 So $\{f_\vf\mid \vf\in\mathfrak{clm}(\L)\}$ is a subspace of the space $\mathcal{SSBD}(\L)$ of skew-symmetric biderivations of $\L$,
 which has dimension 2 if $\L=W(0,-1)$ or $\widetilde W(0,-1)$,
 and dimension 1 otherwise, by Theorem \ref{thm:witt}, Theorem \ref{thm:vir}, Theorem \ref{thm:wab} and Theorem \ref{thm:twab}.

 It is easy to check that $\Theta_1^{(0,-1)}$ and $\widetilde\Theta_1^{(0,-1)}$ are commuting linear maps of $W(0,-1)$ and $\widetilde W(0,-1)$ respectively.
 So the skew-symmetric biderivations constructed through \eqref{eq:clm} using all linear combinations of $\mathrm{id}, \Theta_1^{(0,-1)}$ and $\widetilde\Theta_1^{(0,-1)}$ exhaust $\mathcal{SSBD}(\L)$ in each case respectively.
 This proves the theorem.
\end{proof}

\subsection{Commutative post-Lie algebra}

In this subsection we determine all commutative post-Lie algebra structures on the algebra $\L$. Let us first recall the definition of commutative post-Lie algebra.

\begin{definition}
 Let $(\g,[\cdot,\cdot])$ be a Lie algebra.
 A \textbf{commutative post-Lie algebra structure} on $\g$
 is a bilinear product $\circ$ satisfying that
 \[x\circ y=y\circ x,\quad [x,y]\circ z=x\circ(y\circ z)-y\circ(x\circ z),\quad x\circ[y,z]=[x\circ y,z]+[y,x\circ z]\]
 for any $x,y,z\in\g$.
 The commutative post-Lie algebra structure $\circ$ on $(\g,[\cdot,\cdot])$ is called \textbf{trivial} if $x\circ y=0$ for all $x,y\in\g$.
 We also call $(\g,[\cdot,\cdot],\circ)$ a \textbf{commutative post-Lie algebra}.

\end{definition}

It is easy to check that the bilinear map $f(x,y)=x\circ y$ on $\g$ is a symmetric biderivation.
This observation enables us to classify all commutative post-Lie algebra structures on the algebra $\L$.

\begin{theorem}
 Any commutative post-Lie algebra structures on the algebra $\L$ is trivial unless $\L=\widetilde W(0,1)$.
 Any nontrivial commutative post-Lie algebra structures on the algebra $\widetilde W(0,1)$ is of the form $\Lambda_{a,b,c}, a,b,c\in\C$, defined by
 \[\Lambda_{a,b,c}(I_0+x,I_0+y)= aC_0+bC_1^1+cC_2^1
  \quad\text{for any }x,y\in\widetilde W(0,1)'.\]
\end{theorem}
\begin{proof}
 Denote by $\mathcal{SBD}(\L)$ the space of symmetric biderivations of $\L$.
 By Theorem \ref{thm:witt}, Theorem \ref{thm:vir}, Theorem \ref{thm:wab} and Theorem \ref{thm:twab} we have
 \[\mathcal{SBD}(\L)=\begin{cases}
      \spanc{\psina 0\mid n\in\Z} & \text{ if }\L=W(a,0);\\
      \spanc{\psina 1\mid n\in\Z} & \text{ if }\L=W(a,1);\\
      \spanc{\mathcal A,\mathcal B,\mathcal C} & \text{ if }\L=\widetilde W(0,1);\\
      0                              & \text{ otherwise}.
  \end{cases}\]

 Let $f$ be the symmetric biderivation of $\L$ defined by
 \[f(x,y)=x\circ y\quad\text{for any }x,y\in\L.\]
 So any commutative post-Lie algebra structures on $\L$ is trivial if $\L\neq W(a,0), W(a,1)$, or $\widetilde W(0,1)$.

 If $\L=W(a,0)$,
 then we may write $f=\sum_{n=q}^p\mu_n\psina0$ for some $\mu_n\in\C$,
 $p,q\in\Z$ such that $p\ge q$.
 Therefore, we have by Theorem \ref{thm:wab}(2)
 \begin{align*}
 \sum_{n=q}^p\mu_nI_n&=f(L_1,L_{-1})=f([L_1,L_0],L_{-1})
  =[L_1,L_0]\circ L_{-1}=L_1\circ (L_0\circ L_{-1})-L_0\circ (L_1\circ L_{-1})\\
  &=f\(L_1,f(L_0, L_{-1})\)-f\(L_0,f(L_1, L_{-1})\)=0,
 \end{align*}
 which implies all $\mu_n=0$. So $f=0$.
 Hence any commutative post-Lie algebra structures on $W(a,0)$ is trivial.

 If $\L=W(a,1)$,
 then we may write $f=\sum_{n=q}^p\mu_n\psina1$ for some $\mu_n\in\C$,
 $p,q\in\Z$ such that $p\ge q$.
 Let $k\in\Z$ be such that $k>-a-p-1$, then we have by Theorem \ref{thm:wab}(2)
 \[\sum_{n=q}^p\mu_n(a+n+k+1)I_{n+k}=f(L_1,L_k)=[L_1,L_0]\circ L_{k}
   =f\(L_1,f(L_0, L_{k})\)-f\(L_0,f(L_1, L_{k})\)=0,\]
 which implies all $\mu_n=0$. So $f=0$.
 Hence any commutative post-Lie algebra structures on $W(a,1)$ is trivial.

 If $\L=\widetilde W(0,1)$,
 then we have $f(x,y)\in\mathfrak C$ for any $x,y\in\widetilde W(0,1)$,
 and one may easily check that for any $a,b,c\in\C$, $\Lambda_{a,b,c}$
 actually defines a commutative post-Lie algebra structure on $\widetilde W(0,1)$.
 This proves the theorem.
\end{proof}

\subsection{Transposed $\dt$-Poisson algebra structure}

In this subsection we determine all transposed $\dt$-Poisson algebra structures on the algebra $\L$.
In this subsection we always assume that $\dt$ is a nonzero complex number.
Let us first recall the definition of transposed $\dt$-Poisson algebra.

\begin{definition}
 A \textbf{transposed $\dt$-Poisson algebra} is a vector space $\mathcal P$ equipped with two bilinear operations $\circ$ and $[\cdot,\cdot]$ such that $(\mathcal P,\circ)$ is a commutative associative algebra, $(\mathcal P,[\cdot,\cdot])$ is a Lie algebra and these two operations satisfy the following compatibility condition
 \[\dt z\circ[x,y]=[z\circ x,y]+[x,z\circ y]\quad\text{ for any }x,y,z\in\mathcal P.\]
 A \textbf{transposed $\dt$-Poisson algebra structure} on a Lie algebra $\g$ is a bilinear operation on $\g$ making $\g$ a transposed $\dt$-Poisson algebra.
 A transposed $\dt$-Poisson algebra structure $\circ$ on $\g$ is called \textbf{trivial} if $x\circ y=0$ for any $x,y\in\g$.
\end{definition}

Let $(\g,\circ,[\cdot,\cdot])$ be a transposed $\dt$-Poisson algebra.
It is easy to check that the bilinear map $f(x,y)=x\circ y$ on $\g$ is a symmetric $\frac1\dt$-biderivation.
This fact enables us to classify all transposed $\dt$-Poisson algebra structures on the algebra $\L$.

\begin{theorem}
 (1) Any transposed $2$-Poisson algebra structure on $\W$ is of the form
 \[L_i\circ L_j=\sum_{n\in\Z}\lmd_nL_{n+i+j}\]
 for some $\lmd_n\in\C$.
 For $\dt\ne2,0$, any transposed $\dt$-Poisson algebra structure on $\W$ is trivial.\\
 (2) Any transposed $1$-Poisson algebra structure on $W(a,0)$ is of the form
 \begin{align*}
  L_i\circ L_j=\sum_{n\in\Z}\lmd_nI_{n+i+j},\quad
  L_i\circ I_j=I_i\circ I_j=0
 \end{align*}
 for some $\lmd_n\in\C$.
 For $\dt\ne1,0$, any transposed $\dt$-Poisson algebra structure on $W(a,0)$ is trivial.\\
 (3) Any transposed $1$-Poisson algebra structure on $W(a,1)$ is of the form
 \begin{align*}
  L_i\circ L_j=\sum_{n\in\Z}(a+n+i+j)\lmd_n I_{n+i+j},\quad
  L_i\circ I_j=I_i\circ I_j=0
 \end{align*}
 for some $\lmd_n\in\C$.
 For $\dt\ne1,0$, any transposed $\dt$-Poisson algebra structure on $W(a,1)$ is trivial.\\
 (4) Any transposed $2$-Poisson algebra structure on $W(a,-1)$ is of the form
 \begin{align*}
  L_i\circ L_j=\sum_{n\in\Z}\lmd_nL_{n+i+j}+\sum_{n\in\Z}\mu_nI_{n+i+j},\quad
  L_i\circ I_j=\sum_{n\in\Z}\lmd_nI_{n+i+j},\quad
   I_i\circ I_j=0
 \end{align*}
 for some $\lmd_n,\mu_n\in\C$.
 For $\dt\ne2,0$, any transposed $\dt$-Poisson algebra structure on $W(a,-1)$ is trivial.\\
 (5) Let $a\notin\Z$. Any transposed $2$-Poisson algebra structure
 on $\widetilde W(a,-1)$, is of the form
 \[L_i\circ L_j=\sum_{n\in\Z}\mu_nI_{n+i+j},\quad
  L_i\circ I_j=I_i\circ I_j=\mathfrak C\circ\widetilde W(a,-1)=0\]
 for some $\mu_n\in\C$.
 For $\dt\ne2,0$, any transposed $\dt$-Poisson algebra structure on $\widetilde W(a,-1)$ is trivial.\\
 (6) Any transposed $2$-Poisson algebra structure
 on $\widetilde W(0,1)$ is of the form
 \[(\lmd I_0+x)\circ(\mu I_0+y)=\lmd\mu\(aC_0+bC_1^1+cC_2^1\),
   \quad \lmd,\mu\in\C, x,y\in \widetilde W(0,1)',\]
 for some $a,b,c\in\C$.
 For $\dt\ne2,0$, any transposed $\dt$-Poisson algebra structure on $\widetilde W(0,1)$ is trivial.\\
 (7) Let $\L$ be any of the Virasoro algebra $\V$, the $W$-algebras $W(a,b)$ or $\widetilde W(a,b)$ except the ones listed in (2), (3), (4), (5) and (6).
 Then any transposed $\dt$-Poisson algebra structure on $\L$ is trivial.
\end{theorem}
\begin{proof}
 By Theorem \ref{thm:witt}, Theorem \ref{thm:vir}, Theorem \ref{thm:wab} and Theorem \ref{thm:twab},
 the space of all symmetric $\frac1\dt$-biderivations of $\L$ equals to
 \[\begin{cases}
    \spanc{\theta_n\mid n\in\Z} &\quad\text{if }\L=\W, \dt=2;\\
    \spanc{\psina 0\mid n\in\Z} &\quad\text{if }\L=W(a,0), \dt=1;\\
    \spanc{\psina 1\mid n\in\Z} &\quad\text{if }\L=W(a,1), \dt=1;\\
    \spanc{\Phi_n,\Psi_n\mid n\in\Z} &\quad\text{if }\L=W(a,-1), \dt=2;\\
    \spanc{\widetilde\Psi_n\mid n\in\Z} &
        \quad\text{if }\L=\widetilde W(a,-1)\text{ and }a\notin\Z, \dt=2;\\
    \spanc{\mathcal A,\mathcal B,\mathcal C} &
      \quad\text{if }\L=\widetilde W(0,1), \dt=2;\\
    0 & \quad\text{otherwise.}
 \end{cases}\]
 Then the theorem follows.
\end{proof}

\bigskip

\textbf{Acknowledgments:}
C. Xu is supported by the National Natural Science Foundation of China (No. 12261077).

\end{document}

%% file: newcommands.tex
\def\({\left(}
\def\){\right)}
\def \<{\langle}
\def \>{\rangle}
\def\al{\alpha}

\def\dt{\delta}

\def\vf{\varphi}
\def\lmd{\lambda}
\def\g{\mathfrak{g}}

\def\V{\mathcal V}
\def\W{\mathcal W}
\def\L{\mathcal L}

\def\ad{\mathrm{ad}}
\newcommand{\C}{\mathbb C}

\newcommand{\Z}{\mathbb{Z}}

\newcommand{\spanc}[1]{\mathrm{span}_{\C}\left\{#1\right\}}
\newcommand{\ov}[1]{\overline{#1}}

\def\wab{W(a,b)}
\def\twab{\widetilde{W}(a,b)}

\newcommand{\der}[1]{\mathfrak{Der}(#1)}
\newcommand{\derdt}[1]{\mathfrak{Der}_\dt(#1)}
\newcommand{\dbd}[1]{\mathfrak{BD}^{[\dt]}\left(#1\right)}
\newcommand{\inn}[1]{\mathrm{Inn}(#1)}
\newcommand{\psina}[1]{\Psi_n^{(a,#1)}}